\documentclass[11pt,reqno]{amsart}

\usepackage[latin1]{inputenc}
\usepackage{amsmath}
\usepackage{amsfonts}
\usepackage{amssymb}
\usepackage{graphics}
\usepackage{amssymb,amsmath,amsthm,amscd,latexsym,verbatim,graphicx,amsfonts}

\newtheorem{theorem}{Theorem}[section]
\newtheorem{lemma}[theorem]{Lemma}

\newtheorem{remark}[theorem]{Remark}

\newcommand{\bbZ}{\mathbb{Z}}
\newcommand{\bbR}{\mathbb{R}}

\author[M.\ Luki\'c]{Milivoje Luki\'c}
\address{M.\ Luki\'c: Emory University, Atlanta, GA~30322, USA}
\email{milivoje.lukic@emory.edu}
\thanks{M.\ L.\ was supported in part by NSF grants DMS--2154563 and DMS--2453758.}

\author[B.\ Simanek]{Brian Simanek}
\address{B. Simanek: Department of Mathematics, Baylor University, Waco, TX, USA}
\email{brian\_simanek@baylor.edu}
\thanks{B.\ S.\ was supported by Simons Foundation Collaboration grant 707882.}

\title{Upper bounds on eigenvalue spacing for decaying potentials}
\date\today

\begin{document}

\dedicatory{Dedicated to Barry Simon on the occasion of his 80th birthday}

\begin{abstract}
We study decaying half-line Schr\"odinger operators and the local eigenvalue spacing of their Dirichlet restrictions. While absolutely continuous spectrum is strongly associated with bulk universality and clock behavior, singular spectral measures can correspond to varied local behaviors. 
In this work, the rate of decay of the potential is shown to give upper bounds for the spacing of Dirichlet eigenvalues on finite intervals.
\end{abstract}

\maketitle

\section{Introduction}

We consider a half-line Schr\"odinger operator $H = -\frac{d^2}{dx^2} + V$ with a real-valued potential $V: (0,\infty) \to \bbR$ which is locally integrable, i.e.,
\[
\int_0^x \lvert V(t) \rvert \,dt < \infty, \qquad \forall x > 0.
\]
The goal of this work is to relate properties of the potential $V$ with the spectra of the Schr\"odinger operators $H_X =-\frac{d^2}{dx^2} + V$ on $L^2((0,X))$ with Dirichlet boundary conditions at $0$ and $X$, for $X \in (0,\infty)$.  
These operators have only discrete spectrum, and our main results will provide upper bounds on the eigenvalue spacing of $H_X$ in terms of the growth rate of the integral
\[
\int_0^x \lvert V(t)\rvert dt
\]
as a function of $x$. As a general principle, it is qualitatively true that greater regularity in the spacing between eigenvalues of $H_X$ is indicative of absolutely continuous spectrum of $H$, while the presence of singular spectrum of $H$ yields greater clustering of the eigenvalues of $H_X$. 

A common way to study spectral properties of differential operators is to discretize them.  Analogs of Schr\"odinger operators in a discrete setting are often found in the theory of orthogonal polynomials, either on the real line or on the unit circle.  In these settings, the analog of the potential $V$ is the sequence of recursion coefficients from the recursion relation satisfied by the orthogonal polynomials.  In \cite{FineIV}, Last and Simon provided estimates on the spacing between zeros of orthogonal polynomials on the real line (which are eigenvalues of self-adjoint matrices with simple spectrum) in terms of decay properties of the recursion coefficients.  In the paper \cite{Sim20}, the second author provided estimates on the spacing between zeros of paraorthogonal polynomials (which are also eigenvalues of unitary matrices with simple spectrum) in terms of decay properties of the sequence of recursion coefficients.  Some of the results in \cite{Sim20} were refinements of earlier results by Golinskii from \cite{Gol02}.  Further results relating zero spacings of paraorthogonal polynomials with the spectrum of the underlying operator can be found in \cite{BS20}.  In particular, the papers \cite{Br11,BS20} show that the aforementioned connection between regularity in eigenvalue spacing and regularity of the spectrum is only a qualitative phenomenon and not a strict rule.

We can now state our main theorem about eigenvalues of $H_X$, which is the analog of the main result from \cite{Sim20} in the setting of half-line Schr\"odinger operators.

\begin{theorem}\label{thm1}
Fix a real-valued locally integrable potential $V$. Fix $a > 0$. There exists a function $h$ of order
\[
h(x) = O \left( \frac 1x \left(1+\int_0^x \lvert V(t) \rvert dt \right) \right), \quad x \to\infty
\]
such that  for any  interval $[\alpha,\beta] \subset [a,\infty)$ of size $\lvert \beta - \alpha \rvert \ge h(X)$ and any $X > 0$, the operator $H_X$ has at least one eigenvalue in $[\alpha^2/4,\beta^2/4]$.
\end{theorem}

The $1+$ in the formula for $h$ in Theorem \ref{thm1} is needed to accommodate the case $V=0$, otherwise it can be ignored. In the case $V \in L^1((0,\infty))$, we have $h(x) = O(1/x)$. For $V \in L^1((0,\infty))$, the canonical spectral measure is purely absolutely continuous on $(0,\infty)$ with positive continuous density, so by \cite{Malt10} (see also \cite{FineIV}), a bulk universality limit holds uniformly on compact intervals, which implies uniform clock spacing with asymptotic distance $1/(2X\sqrt E)$, which matches our $h(x) = O(1/x)$.

Bulk universality holds a.e.\ on an essential support of the absolutely continuous spectrum (see \cite{AvilaLastSimon,ELS25,ELW25}), but when singular spectrum is present, one can observe less regular spacing of the eigenvalues. Accordingly, the more interesting application of Theorem~\ref{thm1} is to $V \notin L^1((0,\infty))$. Decay slower than $L^1$ leads to spectral transitions which have been heavily studied  
\cite{Kiselev96,Kiselev98,Remling98,RemlingPAMS1998,RemlingDuke2000,ChristKiselev98,CAK2001CMP,ChristKiselev,ChristKiselev02,DeiftKillip,KillipSimon,LukicWang};
see also surveys \cite{DenisovKiselev,Killip}. These results guarantee presence of absolutely continuous spectrum for $V\in L^2((0,\infty))$ and restrict the allowed singular continuous spectrum for decay conditions between $L^1$ and $L^2$. Decay assumptions on the potential are often stated using amalgamated norms for $p\in (1,\infty)$,
\[
\lVert V \rVert_{\ell^p(L^1)} = \left( \sum_{n=0}^\infty v_n^p \right)^{1/p},
\]
where 
\[
v_n = \int_n^{n+1}\lvert V(x) \rvert dx,
\]
and their weak counterparts,
\[
\lVert V \rVert_{\ell^p_w(L^1)} = \sup_{s\in (0,\infty)} s \left( \# \{ n \mid v_n > s \} \right)^{1/p}.
\]
These norms have the benefit of imposing weaker decay assumptions than $L^1$, without imposing stronger local assumptions than local integrability.

\begin{remark}\label{remark12}
Although the growth rate of $\int_0^x |V(t)|dt$ may seem like a power law condition on $V$, we note that it is controlled by $L^p$ and amalgamated $L^p$ conditions and their weak counterparts. For any $p \in (1,\infty)$:
\begin{enumerate}
\item[(a)] If $V \in L^p((0,\infty))$ or $V \in \ell^p(L^1)$, then
\begin{equation}\label{eqnConditionp}
\int_0^x \lvert V(t) \rvert\,dt = o(x^{1-1/p}), \quad x\to\infty.
\end{equation}
\item[(b)] If $V \in L^p_w((0,\infty))$ or $V \in \ell^p_w(L^1)$, then
\begin{equation}\label{eqnWeakCondition}
\int_0^x \lvert V(t) \rvert dt = O(x^{1-1/p}), \quad x \to \infty.
\end{equation}
\end{enumerate}
\end{remark}

The proof of Theorem~\ref{thm1} uses Pr\"ufer variables, which quantify how a perturbation (such as a decaying potential) affects eigensolutions. Pr\"ufer variables have been introduced in many one-dimensional settings, see \cite{KLS,KRS,SimonOPUC2,LukicOng,LukicSukhtaievWang}.

In Section~\ref{sect2} we provide proofs of the statements in Remark~\ref{remark12}. In Section~\ref{sect3} we prove Theorem~\ref{thm1}.

\section{Growth rate estimates}\label{sect2}

Recall first that $V\in L^p$ implies $V \in \ell^p(L^1)$, since by H\"older's inequality, $\lvert v_n \rvert^p \le \int_n^{n+1} \lvert V(t)\rvert^p \,dt$.

\begin{lemma}
For $p > 1$, $V \in \ell^p(L^1)$ implies \eqref{eqnConditionp}.
\end{lemma}

\begin{proof}
For any $\epsilon > 0$, we can pick $n_0$ such that $\sum_{n_0}^\infty |v_n|^p < \epsilon^p$. Then by H\"older's inequality, for $n > n_0$,
\[
\int_{n_0}^n |V(t)| \,dt = \sum_{k=n_0}^{n-1} |v_k| \le \left(\sum_{k=n_0}^{n-1} |v_k|^p \right)^{1/p} (n - n_0)^{1-1/p} \le \epsilon n^{1-1/p}.
\]
Since the integral $\int_0^x \lvert V(t) \rvert dt$ is increasing in $x$, this implies  \eqref{eqnConditionp}.
\end{proof}

\begin{lemma}
For $p > 1$, $V \in L^p_w$ implies \eqref{eqnWeakCondition}.
\end{lemma}

\begin{proof}
Denote $\delta = \lVert V \rVert_{L^p_w}$. By definition, for all $s >0$,
\[
\lvert \{ x \in [0,\infty) \mid \lvert V(x) \rvert > s \} \rvert \le \frac{\delta^p}{s^{p}}.
\]
This implies
\begin{align*}
\int_0^{x}|V(t)|dt & = \int_0^{\infty}|\{t\in[0,x]:|V(t)|>s\}|ds\leq\int\min\left\{x,\frac{\delta^p}{s^p}\right\}ds\\
&=\int_0^{\delta x^{-1/p}}xds+\int_{\delta x^{-1/p}}^{\infty}\frac{\delta^p}{s^p}ds=\delta x^{1-1/p}+\frac{\delta}{p-1}x^{1-1/p}
\end{align*}
\end{proof}

\begin{lemma}
For $p > 1$, $v \in \ell^p_w$ implies \eqref{eqnWeakCondition}.
\end{lemma}

\begin{proof}
Denote $\delta = \lVert v \rVert_{\ell^p_w}$. Then by definition,
\[
\# \{ n \mid v_n > s \} \le \frac{\delta^p}{s^{p}}.
\]
For a positive integer $N$,
\begin{align*}
\sum_{n=0}^{N-1} v_n & = \int_0^\infty \# \{n \mid n < N, v_n > s\} ds \le \int_0^\infty \min \{ N, \delta^p / s^p \} ds \\
&=\int_0^{\delta N^{-1/p}}xds+\int_{\delta N^{-1/p}}^{\infty}\frac{\delta^p}{s^p}ds=\delta N^{1-1/p}+\frac{\delta}{p-1}N^{1-1/p}
\end{align*}
Since the integral $\int_0^x \lvert V(t) \rvert dt$ is increasing in $x$, this implies \eqref{eqnWeakCondition}.
\end{proof}

\section{The main proof}\label{sect3}

We will consider the Dirichlet eigensolutions at energy $E = k^2 > 0$, which solve the Schr\"{o}dinger equation
\[
-u''(x)+V(x)u(x)=Eu(x)
\]
and satisfy the initial conditions $u(0)=0, u'(0)=1$. 
The Pr\"{u}fer variables $R$ and $\theta$ are defined so
\begin{align}
u(x)&=R_{k}(x)\sin(k x +\theta_{k}(x))\\
u'(x)&= k R_{k}(x)\cos(k x +\theta_{k}(x))
\end{align}
The Pr\"ufer phase obeys the first order equation
\begin{equation}\label{thetaprime}
\frac{d\theta_{k}}{dx}=\frac{V(x)}{2k}[\cos(2k x+2\theta_{k}(x))-1].
\end{equation}
Define
\begin{equation}\label{fdef}
f_{k}(x)= k x + \theta_{k}(x).
\end{equation}
In analogy with the main results of \cite{Sim20}, the goal is to understand quantities like
\[
f_{\beta}(x)-f_{\alpha}(x)
\]
as $x\rightarrow\infty$.  Our next lemma is concerned only with fixed $x$:

\begin{lemma}\label{lemmaFiniteLength}
Fix $[\alpha,\beta] \subset (0,\infty)$ and $x > 0$. If
\begin{equation}\label{eqnCriterionx}
(\beta - \alpha) x \ge \pi + (\beta^{-1} + \alpha^{-1}) \int_0^x \lvert V(t) \rvert \,dt
\end{equation}
then the operator $H_x$ has an eigenvalue in $[\alpha^2, \beta^2]$.
\end{lemma}

\begin{remark}\label{remark31}
Adding a constant $c$ to the potential $V$ and to the energy levels $\alpha^2$, $\beta^2$ does not effect eigenvalue spacings, but affects \eqref{eqnCriterionx}. Optimizing in $c$ could lead to a better statement for a fixed interval; however, for our intended application to decaying potentials, adding a constant $c$ would increase the growth rate of $h$.
\end{remark}

\begin{proof}
According to \eqref{fdef} and \eqref{thetaprime}, we have
\begin{align*}
f_{\beta}(x)-f_{\alpha}(x)&= (\beta - \alpha) x +\int_0^x\frac{V(t)}{2\beta}
[\cos(\beta t +\theta_{\beta}(t))-1]dt\\
&\qquad-\int_0^x\frac{V(t)}{2\alpha}[\cos(\alpha t +\theta_{\alpha}(t))-1]dt
\end{align*}
Using $\lvert \cos - 1 \rvert \le 2$ and \eqref{eqnCriterionx}, this implies 
\[
f_{\beta}(x) - f_\alpha(x) \ge \pi.
\]
By continuity in $k$, there exists $k \in [\alpha,\beta]$ such that $f_k(x) \in \pi \bbZ$. Therefore, for that value of $k$, $u(x) = 0$, so $k^2 \in \sigma(H_x)$.
\end{proof}

\begin{proof}[Proof of Theorem~\ref{thm1}]
Fix $a > 0$ and define
\[
h(x) = \frac \pi x + \frac 2{ax}\int_0^x \lvert V(t) \rvert \, dt.
\]
Then for $[\alpha,\beta] \subset [a,\infty)$ with $\beta - \alpha \ge h(x)$, we have
\[
(\beta -\alpha)x \ge \pi + \frac 2a \int_0^x \lvert V(t) \rvert \, dt
\]
which, with $\alpha,\beta \ge a$, implies \eqref{eqnCriterionx}. Thus, Lemma~\ref{lemmaFiniteLength} concludes the proof.
\end{proof}


\begin{thebibliography}{14}

\bibitem{AvilaLastSimon}
A.~Avila, Y.~Last, B.~Simon, \emph{Bulk universality and clock spacing of
  zeros for ergodic {J}acobi matrices with absolutely continuous spectrum},
  Anal. PDE {3} (2010), no.~1, 81--108.


\bibitem{Br11} J. Breuer, {\em Sine kernel asymptotics for a class of singular measures}, J. Approx. Theory 163 (2011), no. 10, 1478--1491.

\bibitem{BS20} J. Breuer, E. Seelig, {\em On the spacing of zeros of paraorthogonal polynomials for singular measures}, J. Approx. Theory 259 (2020), 105482, 20 pp.

\bibitem{ChristKiselev98}
{M.~Christ, A.~Kiselev}, {\em Absolutely continuous spectrum for
  one-dimensional {S}chr\"{o}dinger operators with slowly decaying potentials:
  some optimal results}, J. Amer. Math. Soc., 11 (1998), pp.~771--797.

\bibitem{Christ2001MaximalFA}
{M.~Christ, A.~Kiselev}, {\em Maximal functions
  associated to filtrations}, J. Funct. Anal., 179 (2001), pp.~409--425.

\bibitem{CAK2001CMP}
{M.~Christ, A.~Kiselev},  {\em W{KB} and spectral
  analysis of one-dimensional {S}chr\"{o}dinger operators with slowly varying
  potentials}, Comm. Math. Phys., 218 (2001), pp.~245--262.

\bibitem{ChristKiselev}
{M.~Christ, A.~Kiselev}, {\em W{KB} asymptotic
  behavior of almost all generalized eigenfunctions for one-dimensional
  {S}chr\"{o}dinger operators with slowly decaying potentials}, J. Funct.
  Anal., 179 (2001), pp.~426--447.

\bibitem{ChristKiselev02}
{M.~Christ, A.~Kiselev}, {\em Scattering and wave operators for
  one-dimensional {S}chr\"{o}dinger operators with slowly decaying nonsmooth
  potentials}, Geom. Funct. Anal., 12 (2002), pp.~1174--1234.

\bibitem{DeiftKillip} P. Deift, R. Killip, \emph{On the absolutely continuous spectrum of one-dimensional Schr\"odinger operators with square summable potentials}, Comm. Math. Phys. 203 (1999) 341--347.

\bibitem{DenisovKiselev} S. A. Denisov, A. Kiselev, \emph{Spectral properties of Schr\"odinger operators with decaying potentials},
Proc. Sympos. Pure Math., 76, Part 2
American Mathematical Society, Providence, RI, 2007.

\bibitem{ELS25} B. Eichinger, M. Luki\'c, B. Simanek, {\em An approach to universality using Weyl m-functions}, to appear in Ann. Math.

\bibitem{ELW25} B. Eichinger, M. Luki\'c, H. Woracek, {\em Necessary and sufficient conditions for universality limits}, arXiv:2409.18045.

\bibitem{Gol02} L. Golinskii, {\em Quadrature formula and zeros of paraorthogonal polynomials on the unit circle}, Acta Math. Hungar. 96 (2002), 169--186.

\bibitem{Killip}
R. Killip, \emph{Spectral theory via sum rules}, Spectral theory and mathematical physics: a Festschrift in honor of Barry Simon's 60th birthday, 907--930.
Proc. Sympos. Pure Math., 76, Part 2 American Mathematical Society, Providence, RI, 2007.

\bibitem{KillipSimon}
R. Killip, B. Simon,  \emph{Sum rules and spectral measures of Schr\"odinger operators with $L^2$ potentials}, 
Ann. of Math. (2) 170 (2009), no. 2, 739--782.

\bibitem{Kiselev96}
{ A.~Kiselev}, {\em Absolutely continuous spectrum of one-dimensional
  {S}chr\"{o}dinger operators and {J}acobi matrices with slowly decreasing
  potentials}, Comm. Math. Phys., 179 (1996), pp.~377--400.

\bibitem{Kiselev98}
{ A.~Kiselev}, {\em Stability of the absolutely continuous spectrum of the
  {S}chr\"{o}dinger equation under slowly decaying perturbations and a.e.
  convergence of integral operators}, Duke Math. J., 94 (1998), pp.~619--646.

\bibitem{KLS}
{ A.~Kiselev, Y.~Last, B.~Simon}, {\em Modified {P}r\"{u}fer and {EFGP}
  transforms and the spectral analysis of one-dimensional {S}chr\"{o}dinger
  operators}, Comm. Math. Phys., 194 (1998), pp.~1--45.

\bibitem{KRS}
A. Kiselev, C. Remling, B. Simon, \emph{Effective perturbation methods for one-dimensional Schr\"odinger operators}, J. Differential Equations 151 (1999), no. 2, 290--312.

\bibitem{FineIV} Y. Last, B. Simon, {\em Fine structure of the zeros of orthogonal polynomials. IV. A priori bounds and clock behavior}, Comm. Pure Appl. Math. 61 (2008), no. 4, 486--538.

\bibitem{Lukic13}
{M. Lukic}, {\em Schr\"{o}dinger
  operators with slowly decaying {W}igner-von {N}eumann type potentials}, J.
  Spectr. Theory, 3 (2013), pp.~147--169.

\bibitem{LukicOng}
M. Lukic, D. C. Ong, \emph{Generalized Pr\"ufer variables for perturbations of Jacobi and CMV matrices}, 
J. Math. Anal. Appl. 444 (2016), 1490--1514.

\bibitem{LukicSukhtaievWang}
M. Luki\'c, S. Sukhtaiev, X. Wang, \emph{Spectral properties of Schr\"odinger operators with locally $H^{-1}$ potentials}, 
J. Spectr. Theory 14 (2024), 59--120.

\bibitem{LukicWang}
M. Luki\'c, X. Wang, \emph{Modified Jost solutions of Schr\"odinger operators with locally $H^{-1}$ potentials}, 
Nonlinearity 38 (2025) 025011.

\bibitem{Malt10} A. Maltsev, {\em Universality limits of a reproducing kernel for a half-line Schr\"{o}dinger operator and clock behavior of eigenvalues}, Comm. Math. Phys. 298 (2010), no. 2, 461--484.

\bibitem{Remling98}
{C.~Remling}, {\em The absolutely continuous spectrum of one-dimensional
  {S}chr\"{o}dinger operators with decaying potentials}, Comm. Math. Phys., 193
  (1998), pp.~151--170.

\bibitem{RemlingPAMS1998}
{C.~Remling}, {\em Bounds on embedded
  singular spectrum for one-dimensional {S}chr\"{o}dinger operators}, Proc.
  Amer. Math. Soc., 128 (2000), pp.~161--171.

\bibitem{RemlingDuke2000}
{C.~Remling}, {\em Schr\"{o}dinger
  operators with decaying potentials: some counterexamples}, Duke Math. J., 105
  (2000), pp.~463--496.


\bibitem{Sim20} B. Simanek, {\em Zero spacings of paraorthogonal polynomials on the unit circle}, J. Approx. Theory 256 (2020) 105437.

\bibitem{SimonOPUC2}
B. Simon, \emph{Orthogonal polynomials on the unit circle. Part 2},
{American Mathematical Society Colloquium Publications 54, Part 2},
 {American Mathematical Society, Providence, RI}, {2005}.


\end{thebibliography}
\end{document}